\documentclass[11pt,twoside]{amsart}
\usepackage[all]{xy}   
\usepackage {amssymb,latexsym,amsthm,amsmath,mathtools}
\usepackage{enumitem,color}

\long\def\symbolfootnote[#1]#2{\begingroup%
\def\thefootnote{\fnsymbol{footnote}}\footnote[#1]{#2}\endgroup}

\newcommand{\tr}{\ensuremath{{}^t\!}}

\newcommand{\C}{\mathfrak C}

\newcommand{\Fq}{\mathbb F_q}

\makeatletter
\def\imod#1{\allowbreak\mkern10mu({\operator@font mod}\,\,#1)}
\makeatother

\newtheorem{theorem}{Theorem}[section]
\newtheorem{lemma}[theorem]{Lemma}

\newtheorem{proposition}[theorem]{Proposition}
\newtheorem*{theorem*}{Theorem}
\theoremstyle{definition}
\newtheorem{definition}[theorem]{Definition}
\newtheorem{remark}[theorem]{Remark}
\newtheorem{example}[theorem]{Example}

\numberwithin{equation}{section}

\newcommand{\ignore}[1]{}

\newcommand{\mynote}[1]{}
\begin{document}
\setcounter{section}{0}
\title{Asymptotics of commuting probabilities in reductive algebraic groups}
\author{Shripad M. Garge}
\address{Department of Mathematics, Indian Institute of Technology, Powai, Mumbai, 400 076 India.}
\email{shripad@math.iitb.ac.in}
\email{smgarge@gmail.com}
\author{Uday Bhaskar Sharma}
\address{Tata Institute of Fundamental Research, Dr. Homi Bhabha Road, Navy Nagar, Colaba, Mumbai 
- 400005, India}
\email{udaybsharmaster@gmail.com}
\author{Anupam Singh}
\address{IISER Pune, Dr. Homi Bhabha Road, Pashan, Pune 411008 India.} 
\email{anupamk18@gmail.com}
\thanks{The third named author acknowledges support of SERB research grant CRG/2019/000271 towards this work.}


\begin{abstract}
Let $G$ be an algebraic group. For $d\geq 1$, we define the commuting probabilities $cp_d(G) = 
\frac{\dim(\C_d(G))}{\dim(G^d)}$, where $\C_d(G)$ is the variety of 
commuting $d$-tuples in $G$. We prove that for a reductive group $G$ when $d$ is large, 
$cp_d(G)\sim \frac{\alpha}{n}$ where $n=\dim(G)$, and $\alpha$ is the maximal dimension of 
an Abelian subgroup of $G$. For a finite reductive group $G$ defined over the field $\mathbb F_q$, 
we 
show that $cp_{d+1}(G(\mathbb F_q))\sim q^{(\alpha-n)d}$, and give several examples.
\end{abstract}
\maketitle
\section{Introduction}
Commuting probability, also referred as commuting degree, is the probability of finding a 
commuting tuple in a group. The question of determining this for a given group is well studied for 
finite groups and compact groups (see for example~\cite{BFM,ET,FF,GR,HR}). To an interested reader 
we recommend the survey article~\cite{SS2} and the references therein for further reading. Let $G$ 
be a group and $d\geq 1$. Let $\C_d(G) = \{(g_1, \ldots, g_d) \in G^d \mid g_ig_j=g_jg_i, \forall 
1\leq i, j \leq d\}$ (also denoted as $G^{(d)}$ sometimes). The elements of $\C_d(G)$ are called 
commuting $d$-tuples. The commuting probability for a finite group $G$, for $d\geq 1$, is defined 
as 
$$cp_d(G)=\frac{|\C_d(G)|}{|G|^d}.$$
Since $\C_1(G)=G$, and $cp_1(G)=1$, we usually take $d\geq 2$. The commuting probability 
$cp_d(G)$ measures the probability of finding a $d$-tuple of elements of $G$ which commute 
pairwise (we will simply call it a $d$-tuple whereas we mean commuting $d$-tuple). While studying 
the 
commuting probabilities for compact groups, instead of size, one considers the measure of 
the sets involved. In this article, we would like to study the asymptotic value of commuting 
probabilities for algebraic groups and finite groups of Lie type. The notion of commuting 
probability in algebraic groups is introduced by the first-named author in~\cite{Ga} where 
$cp_2$ is defined using the dimension of the subsets involved. We generalise that to define $cp_d$ 
here. 

Let $K$ be an algebraically closed field. Let $G$ be an algebraic group over $K$. The set 
$\C_d(G)$ is an algebraic variety, often called commuting variety in the literature. For $d\geq 
1$, we define the commuting probabilities as follows,
$$cp_d(G)=\frac{\dim(\C_d(G))}{\dim(G^d)} = \frac{\dim(\C_d(G))}{d.\dim(G)}. $$
Clearly, $cp_1(G)=1$, thus, in what follows we take $d\geq 2$. The questions such as if
$\C_d(G)$ is an irreducible variety, is an intense topic of study. 
Richardson~\cite[Theorem C]{Ri} proved that $\C_2(G)$ is an irreducible variety when $G$ is a 
simply connected semisimple algebraic group. The commuting varieties for matrices and Lie algebras 
are well studied (see, for example~\cite{FG,GuSe}). However, our concern here is its dimension. 
In Section~\ref{section-dim-ct}, we get a bound on this using the idea of the branching matrix 
developed in Section~\ref{section-BMalg-gps}. 

Recall that we say $\{a_n\}$ is asymptotic to $\{b_n\}$, as $n$ gets large, if $\displaystyle 
\lim_{n \to\infty}\frac{a_n}{b_n}=1$, and we write $a_n\sim b_n$. Let $G$ be a reductive algebraic 
group of dimension $n$, and maximal dimension of an Abelian subgroup is $\alpha$. In 
Section~\ref{section-alg-gps}, we prove that for a reductive algebraic group $G$, when $d$ gets 
large, $cp_d(G)\sim \frac{\alpha}{n}$ (see Theorem~\ref{theorem-cpd-alg-gps}). In 
Section~\ref{section-fg}, using~\cite[Theorem 3.1]{KPP} where simultaneous conjugacy classes are 
studied for finite groups, we note that for a finite group $G$ the 
commuting probabilities $cp_d(G)\sim m\left(\frac{a}{|G|}\right)^{d-1}$, where $m$ is a constant, 
and $a$ is the maximal size of an Abelian subgroup of $G$. We give an alternate proof of this 
result 
using the ideas developed in this paper. We apply this result on a finite reductive group $G$ 
defined over the field $\mathbb F_q$ to get $cp_d(G(\mathbb F_q)) \sim 
\left(\frac{1}{q^{n-\alpha}}\right)^{d-1}$ up to a constant (see Theorem~\ref{CP-fgl}). The maximal 
Abelian subgroups are known for finite groups of Lie type (see, for 
example~\cite{Vd1,Vd2,Wo1,Wo2,Ba}). For several examples of classical groups, we use this to compute 
the asymptotic value of commuting probabilities, and find tuples of which common centralizer is a 
maximal Abelian subgroup.


\section{Branching matrix for algebraic groups}\label{section-BMalg-gps}
Let $K$ be an algebraically closed field, and $G$ be an algebraic group over $K$. To study 
the commuting probabilities $cp_d(G)$, we introduce branching matrix $B_G$ for $G$. This concept 
will be generalized from that of finite groups given in~\cite{SS, SS2}. The size of branching 
matrix $B_G$ will turn out to be the number of $z$-classes of commuting tuples, and entries of 
$B_G$ will be a measure of different conjugacy classes, which are in the same $z$-class. In this 
section, we define these notions for algebraic groups and prove the relation between $B_G$ and 
the commuting probabilities. 

\subsection{$z$-classes of tuples}
This notion is a generalization of the similar concept studied for finite groups and algebraic 
groups (see for example~\cite{BS,GS}). We define an equivalence relation, namely $z$-equivalence, 
on the set of 
commuting $d$-tuples $\C_d(G)$, for $d\geq 1$, as follows. 
\begin{definition}
The tuples $(g_1,\ldots, g_d)$ and $(h_1,\ldots, h_d)\in \C_d(G)$ are said to be $z$-equivalent if 
$\mathcal Z_G(g_1,\ldots, g_d)$ and $\mathcal Z_G(h_1,\ldots, h_d)$ are conjugate in $G$, where 
$\displaystyle \mathcal Z_G(g_1,\ldots, g_d) = \bigcap_{i=1}^d \mathcal Z_G(g_i)$ denotes the 
intersection of centralizers of $g_1,\ldots, g_d$ in $G$. We call the corresponding equivalence 
classes the {\bf $z$-classes of $d$-tuples}. 
\end{definition}
\noindent Notice that, we can make $G$ act on $\C_d(G)$ by conjugation component wise, thus giving 
rise to the conjugacy classes of $d$-tuples. Hence, a $z$-class of $d$-tuple is a union of those 
conjugacy classes of $d$-tuples for which the corresponding common centralizers are conjugate 
within $G$. For $d=1$, this definition coincides with the usual notion of {\bf $z$-classes} in 
$G$, thus $z$-classes of $1$-tuples are simply the $z$-classes. The number of $z$-classes is known 
to be finite for a reductive algebraic group. This was proved by Steinberg (see Section 3.6 
Corollary 1 to Theorem 2~\cite{St}), and is further explored over fields of type $(F)$ 
in~\cite{GS}. 
However, this number could be infinite for a more general algebraic group, for example upper 
triangular matrix 
group (see~\cite[Theorem 1.2]{Bh}). For more on $z$ classes, we refer an interested reader to the 
survey article~\cite{BS}. The notion of $z$-classes can be defined among all tuples.
\begin{definition}
We define $z$-equivalence on $\C(G)=\displaystyle\bigcup_{d\geq 1} \C_d(G)$ as follows. The tuples 
$(g_1,\ldots, g_e)\in \C_e(G)$ and $(h_1,\ldots, h_f)\in \C_f(G)$ are said to be {\bf 
$z$-equivalent} if $\mathcal Z_G(g_1,\ldots, g_e)$ and $\mathcal Z_G(h_1,\ldots, h_f)$ are 
conjugate in $G$. We call the equivalence classes in $\C(G)$, the {\bf $z$-classes of tuples}. 
\end{definition}
\noindent The number of $z$-classes of tuples is obviously finite when $G$ is a finite group, but 
its finiteness for a reductive algebraic group requires some work. We begin with the following,
\begin{proposition}\label{finite-cdg}
Let $G$ be a reductive algebraic group. Then, there are finitely many $z$-classes of $d$-tuples 
(i.e., $z$-classes in $\C_d(G)$), for any $d\geq 1$. 
\end{proposition}
\begin{proof}
For $d=1$, this is a result due to Steinberg as mentioned earlier in this section. We prove 
this for $d=2$. Let $(g_1,g_2)\in \C_2(G)$, write $g_1=s_1u_1$ and $g_2=s_2u_2$, its 
Jordan decomposition. First, we consider the case when $s_1=s_2=s$. In this case,
$$\mathcal Z_G(g_1,g_2)= \mathcal Z_G(g_1)\cap \mathcal Z_G(g_2) = \mathcal Z_{\mathcal 
Z_G(s)}(u_1)\cap \mathcal Z_{\mathcal Z_G(s)}(u_2)= \mathcal Z_{\mathcal Z_G(s)}(u_1,u_2).$$
Since, $\mathcal Z_G(s)$ is a reductive group, and there are only finitely many unipotent classes 
in such groups, this number is finite. 

Now, we need to deal with the general case. Note that since $g_1$ and $g_2$ commute, the 
elements $s_1,s_2, u_1$ and $u_2$ commute pairwise. This is because of Jordan decomposition which 
also gives us that $s_i, u_i$ are polynomials in $g_i$. Hence $u_1, u_2, s_2\in \mathcal Z_G(s_1)$, 
further, $u_1, u_2\in \mathcal Z_{\mathcal Z_{G}(s_1)}(s_2)$. Now, 
$$\mathcal Z_G(g_1,g_2)= \mathcal Z_G(s_1)\cap \mathcal Z_G(u_1) \cap \mathcal Z_G(s_2) \cap 
\mathcal Z_G(u_2)= \mathcal Z_{\mathcal Z_{\mathcal Z_G(s_1)}(s_2)}(u_1,u_2).$$
Once again the group, $\mathcal Z_G(s_1)$ is reductive and $\mathcal Z_{\mathcal Z_{G}(s_1)}(s_2)$ 
as well. Since there are only finitely many conjugacy classes of unipotents in a reductive group, 
we get the finiteness of $z$-classes. The proof for $d$-tuples can be done similarly by looking at 
repeated centralizers of the semisimple components.
\end{proof}
\noindent Now we prove,
\begin{proposition}\label{finite-branches}
Let $G$ be a reductive algebraic group. Then, the number of $z$-classes of tuples in $G$ (i.e., 
$z$-classes in $\C(G)$) is finite.
\end{proposition}
\begin{proof}
Let $(g_1,\ldots, g_d)\in \C(G)$ and $\mathcal Z_G(g_1, \ldots, g_d)= \displaystyle 
\bigcap_{i=1}^d \mathcal Z_G(g_i)$. We can write,
\begin{eqnarray*}
\mathcal Z_G(g_1) &\supset& \mathcal Z_G(g_1, g_2) =  \mathcal Z_{\mathcal 
Z_G(g_1)}(g_2) \supset \cdots \supset \mathcal Z_G(g_1, \ldots, g_i) = \mathcal Z_{\mathcal 
Z_G(g_1, \ldots, g_{i-1})}(g_i)\supset \cdots \\ 
& & \cdots \supset \mathcal Z_G(g_1, \ldots, g_d) = \mathcal Z_{\mathcal 
Z_G(g_1, \ldots, g_{d-1})}(g_d).
\end{eqnarray*}
Note that for large enough $d$, this series will end in an Abelian group. From 
Proposition~\ref{finite-cdg}, each step in the above chain has finitely many choices. Further, 
the length of such a chain is finite. Since, a strict inclusion in the above chain can come for 
one of the following two reasons: either the subgroup is connected then dimension of the subgroup 
decreases or if the subgroup is not connected then it is of finite index (being algebraic 
subgroup). Thus, we have finitely many $z$-classes of tuples. 
\end{proof}
\noindent The number of $z$-classes of $\C(G)$ will turn out to be the size of ``branching 
matrix'' of $G$ which we will define next.

\subsection{Branching matrix}\label{subsection-BM}
Now, we define the {\bf branching matrix} $B_G$ for an algebraic group $G$. The notion 
of branching matrix, and its relation with the commuting tuples for finite groups, has been 
explored in~\cite{Sh,SS,SS2}. The rows and columns of this matrix correspond to $z$-classes of 
tuples in $G$ (i.e. $z$-classes in $\C(G)$). We begin with fixing a convention where the 
$z$-classes of the group, i.e, for $d=1$, will be written first. Furthermore, we take the first 
entry to be the $z$-class of identity (equivalently, any central element) of $G$. Then, we take 
the $z$-classes of $2$-tuples, $3$-tuples and so on. Fix an indeterminate $\psi$. The 
entries of the matrix $B_G$ are monomials in the variable $\psi$, and are defined as follows. For a 
$z$-class of a commuting $d$-tuple $(g_1,\ldots, g_d)$ we look at the group $\mathcal 
Z_G(g_1,\ldots, g_d):=H$, and compute its $z$ classes (i.e., of $1$-tuples). Notice that 
$\C(H)\subset \C(G)$. Suppose an $r$-tuple $(x_1,\ldots,x_r)\in \C_r(G)$ appears as a $z$-class of 
$H= \mathcal Z_G(g_1,\ldots, g_d)$. Then, in the column corresponding to the $z$-class of 
$(g_1,\ldots, g_d)$, we put the entry 
$$\psi^{\dim zcl(x_1,\ldots, x_r) -\dim cl(x_1,\ldots, x_r)}$$  
in $B_G$ where $zcl(x_1,\ldots, x_r)$ and $cl(x_1,\ldots, x_r)$ denote the $z$-class and conjugacy 
class in $H=\mathcal Z_G(g_1,\ldots, g_d)$ of the tuple $(x_1,\ldots, x_r)$, respectively. 
Equivalently, there exists an element $y\in H$ such that $\mathcal Z_H(y)=\mathcal Z_G(x_1, 
\ldots, x_r)$, and $zcl(x_1,\ldots, x_r):=zcl(y)$ and $cl(x_1,\ldots, x_r):=cl(y)$. If an
$r$-tuple $(x_1,\ldots,x_r)\in \C_r(G)$ does not appear as a $z$-class of $H= \mathcal 
Z_G(g_1,\ldots, g_d)$, then we enter $0$ in $B_G$. Thus, to compute the branching matrix of an 
algebraic group we need to follow the steps mentioned below:
\begin{enumerate}
\item To begin with, we compute the $z$-classes in $G$, say the representatives for these classes 
are $\{z_1=e, z_2, \ldots, z_r \}$. The first column corresponds to the identity (as per our 
convention). Now, to obtain the entries in first column, we compute the $z$-classes in $\mathcal 
Z_G(1)=G$, and enter $\psi^{\dim zcl(z_i)-\dim cl(z_i)}$ as entries.
\item Then, we fill the columns $2$ to $r$ corresponding to the non-identity $z$-classes, i.e., 
for $z_2, z_3, \ldots, z_r$. For example, to get the second column we need to compute the 
$z$-classes 
within $\mathcal Z_G(z_2)$. We fill the entries in $B_G$, as per the formula, if the $z$-classes 
match with the ones from that of $G$ obtained in the previous step, else we create a new 
row and a new column as this would correspond to a $2$-tuple of $G$ (i.e., it would give rise 
to some $z$-class of $\C_2(G)$). We do this process for all $\mathcal Z_G(z_i)$, and 
whenever we find a new type of $2$ tuple, we add a new row and column at the end.  
\item After finishing the previous step for all $z$-classes (of $1$-tuples), we look at the new 
ones obtained in those steps. These new ones correspond to $2$-tuples which will give rise to new 
centralizer subgroups, namely, the intersection of centralizers. We compute the $z$-classes in 
these new centralizer subgroups to fill the corresponding column as explained in the previous 
step, and possibly obtain some new types of $3$-tuples. We continue this process till we get no 
more new tuples. The process ends when we get to the Abelian centralizers.  
\end{enumerate}
Notice that there is no guarantee that $B_G$ is a finite size matrix at the moment. To understand 
these steps better, we work out some examples with the help of following, 
\begin{remark}\label{computingBG}
To understand the dimension of $cl(g)$, for $g\in G$, we can look at the dimension of $G/\mathcal 
Z_G(g)$. However, to understand the dimension of $zcl(g)$ we need to look at the set $zcl(g) = 
\bigcup_{t} cl(t)$ where $\mathcal Z_G(t)$ is conjugate to $\mathcal Z_G(g)$. By associating 
$\mathcal Z_G(t)$ to each $t$ we need to understand various conjugates of $\mathcal Z_G(g)$, and 
those $t$ (up to conjugacy) for which $\mathcal Z_G(t)= \mathcal Z_G(g)$. This amounts to 
understanding the set $G/N_G(\mathcal Z_G(g)) \bigcup \{x\in cl(G) \mid \mathcal Z_G(x)= \mathcal 
Z_G(g)\}$. Thus, 
$$\dim zcl(g)-\dim cl(g) = -\dim(N_G(\mathcal Z_G(g))) + \dim \{x\in cl(G) \mid \mathcal Z_G(x)= 
\mathcal Z_G(g)\} + \dim \mathcal Z_G(g).$$
We will take help of this equation in the following computations. 
\end{remark}
\begin{example} For the algebraic group $GL_2$ over $K$ the number of $z$-classes is $3$ 
given by $I, \begin{pmatrix} 1&1\\0&1 \end{pmatrix}$ and $\begin{pmatrix}\lambda_1 & \\& 
\lambda_2 \end{pmatrix}$ where $\lambda_1\neq \lambda_2$. Notice that there are no $z$-classes 
for $d\geq 2$ tuples as all non-trivial centralizers are Abelian. The branching matrix is as 
follows:
$$B_{GL_2} = \begin{pmatrix} \psi & 0 &0 \\ \psi &\psi^2&0 \\ \psi^2& 0 & \psi^2\end{pmatrix}.$$
To get the first column we compute $z$-classes in $\mathcal Z_{GL_2}(I)=GL_2$, to get the 
second column we compute the $z$-classes in $\mathcal Z_{GL_2} \begin{pmatrix} 1&1\\0&1 
\end{pmatrix}= \left\{\begin{pmatrix} a&b\\0&a \end{pmatrix}\mid a\in K^*, b\in K\right\}$, and to 
get the third column we compute the $z$-classes in $\mathcal Z_{GL_2} \begin{pmatrix}\lambda_1 & 
\\& \lambda_2 \end{pmatrix}$, which is the diagonal group.
\end{example}
\begin{example}
Let us look at $GL_3$. The number of $z$-classes in $GL_3(K)$ is $6$ and there are no higher 
tuples. We write them in the following order: 
$$\left\{aI_3, 
\begin{pmatrix}a&1&\\&a&\\&&a\end{pmatrix}, \begin{pmatrix}a&&\\&a&\\&&b\end{pmatrix}
, \begin{pmatrix}a&1&\\&a&1\\&&a\end{pmatrix}, 
\begin{pmatrix}a&1&\\&a&\\&&b\end{pmatrix}, 
\begin{pmatrix}a&&\\&b&\\&&c\end{pmatrix} \right\}
$$
where $a, b$ are distinct and non-zero. The branching matrix $B_{GL_3}$ is: 
$$\begin{pmatrix}
\psi & 0 &0 & 0 & 0 & 0\\
\psi & \psi^2 & 0 & 0 & 0 & 0\\
\psi^2 & 0 & \psi^2 & 0 & 0 & 0\\
\psi & \psi^2 & 0 & \psi^3 & 0 & 0\\
\psi^2 & \psi^3 & \psi^2 & 0 & \psi^3 & 0\\
\psi^3 & 0 & \psi^3 & 0 & 0 & \psi^3
\end{pmatrix}
.$$
\end{example}

\begin{example}
For the group $GL_4$, we have $14$ $z$-classes of $1$-tuples. In addition to these, there are 
four more new $z$-classes of $2$-tuples (indicated in blue colour in the branching matrix) and one more new $z$-class of triples (indicated in red colour). The $z$-class of 
triple has its centralizer, an Abelian subgroup of maximal dimension. The representative of 
$z$-classes are as follows:
$$aI_4, \left(\begin{smallmatrix}a&1&&\\&a&&\\&&a&\\&&&a\end{smallmatrix}\right), 
\left(\begin{smallmatrix} a&1&&\\&a&&\\&&a&1\\&&&a\end{smallmatrix}\right), 
\left(\begin{smallmatrix}a&1&&\\&a&1&\\&&a&\\&&&a\end{smallmatrix}\right),\left(\begin{smallmatrix}
aI_3 & \\ & b\end{smallmatrix}\right), 
\left(\begin{smallmatrix}a&1&&\\&a&&\\&&a&\\&&&b\end{smallmatrix}\right),
\left(\begin{smallmatrix}aI_2&\\&bI_2\end{smallmatrix}\right), 
\left(\begin{smallmatrix}a&1&&\\&a&&\\&&b&\\&&&b\end{smallmatrix}\right), $$
$$
\left(\begin{smallmatrix}a&&&\\&a&&\\&&b&\\&&&c\end{smallmatrix}\right), 
\left(\begin{smallmatrix}a&1&&\\&a&1&\\&&a&1\\&&&a\end{smallmatrix}\right),
\left(\begin{smallmatrix}a&1&&\\&a&1&\\&&a&\\&&&b\end{smallmatrix}\right),
\left(\begin{smallmatrix}a&1&&\\&a&&\\&&b&1\\&&&b\end{smallmatrix}\right),
\left(\begin{smallmatrix}a&1&&\\&a&&\\&&b&\\&&&c\end{smallmatrix}\right),
\left(\begin{smallmatrix}a&&&\\&b&&\\&&c&\\&&&d\end{smallmatrix}\right),
$$
and, for $2$ and $3$-tuples 
$$
\left( \left(\begin{smallmatrix}a&&1&\\&a&&1\\&&a&\\&&&a\end{smallmatrix}\right), 
\left(\begin{smallmatrix}b&&c&1\\&b&&c\\&&b&\\&&&b\end{smallmatrix}\right)\right),
\left( \left(\begin{smallmatrix}a&&1&\\&a&&1\\&&a&\\&&&a\end{smallmatrix}\right), 
\left(\begin{smallmatrix}b&&c&\\&b&&d\\&&b&\\&&&b\end{smallmatrix}\right)\right),
\left( \left(\begin{smallmatrix}a&1&&\\&a&&\\&&a&\\&&&a\end{smallmatrix}\right), 
\left(\begin{smallmatrix}b&&&1\\&b&&\\&&b&\\&&&b\end{smallmatrix}\right)\right),$$
$$
\left( \left(\begin{smallmatrix}a&&&1\\&a&&\\&&a&\\&&&a\end{smallmatrix}\right), 
\left(\begin{smallmatrix}b&&&\\&b&&1\\&&b&\\&&&b\end{smallmatrix}\right)\right),
\left( \left(\begin{smallmatrix}a&&1&\\&a&&1\\&&a&\\&&&a\end{smallmatrix}\right), 
\left(\begin{smallmatrix}b&&c&\\&b&&d\\&&b&\\&&&b\end{smallmatrix}\right),\left(\begin{smallmatrix}
u&&&1\\&u&1&\\&&u&\\&&&u\end{smallmatrix}\right)\right).$$
The branching matrix, with the row and column indexing as above, is 
$$\left(
\begin{smallmatrix}
\psi & 0 & 0 & 0 & 0 & 0 & 0 & 0 & 0 & 0 & 0 & 0 & 0 & 0 & 0 & 0 & 0 & 0 & 0 \\ 
\psi & \psi^{2} & 0 & 0 & 0 & 0 & 0 & 0 & 0 & 0 & 0 & 0 & 0 & 0 & 0 & 0 & 0 & 0 & 0 \\ 
\psi & 0 & \psi^{2} & 0 & 0 & 0 & 0 & 0 & 0 & 0 & 0 & 0 & 0 & 0 & 0 & 0 & 0 & 0 & 0 \\ 
\psi & \psi^2 & 0 & \psi^{3} & 0 & 0 & 0 & 0 & 0 & 0 & 0 & 0 & 0 & 0 & 0 & 0 & 0 & 0 & 0 \\ 
\psi^2 & 0 & 0 & 0 & \psi^{2} & 0 & 0 & 0 & 0 & 0 & 0 & 0 & 0 & 0 & 0 & 0 & 0 & 0 & 0 \\ 
\psi^2 & \psi^3 & 0 & 0 & \psi^2 & \psi^{3} & 0 & 0 & 0 & 0 & 0 & 0 & 0 & 0 & 0 & 0 & 0 & 0 & 0 \\ 
\psi^2 & 0 & 0 & 0 & 0 & 0 & \psi^{2} & 0 & 0 & 0 & 0 & 0 & 0 & 0 & 0 & 0 & 0 & 0 & 0 \\ 
\psi^2 & \psi^3 & 0 & 0 & 0 & 0 & \psi^2 & \psi^{3} & 0 & 0 & 0 & 0 & 0 & 0 & 0 & 0 & 0 & 0 & 0 \\ 
\psi^3 & 0 & 0 & 0 & \psi^3 & 0 & \psi^3 & 0 & \psi^{3} & 0 & 0 & 0 & 0 & 0 & 0 & 0 & 0 & 0 & 0 \\ 
\psi & \psi^2 & \psi^3 & \psi^3 & 0 & 0 & 0 & 0 & 0 & \psi^{4} & 0 & 0 & 0 & 0 & \psi^4 & 0 & \psi^3 
& \psi^3 & 0 \\ 
\psi^2 & \psi^3 & 0 & \psi^4 & \psi^2 & \psi^3 & 0 & 0 & 0 & 0 & \psi^{4} & 0 & 0 & 0 & 0 & 0 & 
\psi^4 & \psi^4 & 0 \\ 
\psi^2 & \psi^3 & \psi^4 & 0 & 0 & 0 & \psi^2 & \psi^3 & 0 & 0 & 0 & \psi^{4} & 0 & 0 & 0 & \psi^4 & 
0 & 0 & 0 \\ 
\psi^3 & \psi^4 & 0 & 0 & \psi^3 & \psi^4 & \psi^3 & \psi^4 & \psi^3 & 0 & 0 & 0 & \psi^{4} & 0 & 0 
& 0 & 0 & 0 & 0 \\ 
\psi^4 & 0 & 0 & 0 & \psi^4 & 0 & \psi^4 & 0 & \psi^4 & 0 & 0 & 0 & 0 & \psi^{4} & 0 & 0 & 0 & 0 & 0 
\\ 
0 & {\color{blue}\psi} & {\color{blue}\psi^2} & 0 & 0 & 0 & 0 & 0 & 0 & 0 & 0 & 0 & 0 & 0 & 
{\color{blue}\psi^{3}} & 0 & 0 & 0 & 0 \\ 
0 & {\color{blue}\psi^2} &{\color{blue} \psi^3} & 0 & 0 & 0 & 0 & 0 & 0 & 0 & 0 & 0 & 0 & 0 & 0 & 
{\color{blue}\psi^{3}} & 0 & 0 & 0 \\ 
0 & {\color{blue}\psi} & 0 & 0 & 0 & 0 & 0 & 0 & 0 & 0 & 0 & 0 & 0 & 0 & 0 & 0 & 
{\color{blue}\psi^{3}} & 0 & 0 \\ 
0 & {\color{blue}\psi}& 0 & 0 & 0 & 0 & 0 & 0 & 0 & 0 & 0 & 0 & 0 & 0 & 0 & 0 & 0 & 
{\color{blue}\psi^{3}} & 0 \\ 
0 & 0 & 0 & 0 & 0 & 0 & 0 & 0 & 0 & 0 & 0 & 0 & 0 & 0 & {\color{red}\psi^4} & {\color{red}\psi^4} & 
{\color{red}\psi^3} & {\color{red}\psi^3} & {\color{red}\psi^{5}} 
\end{smallmatrix}\right).$$
\end{example}
\vskip2mm
\noindent We list some useful properties of the branching matrix $B_G$ when $G$ is an algebraic 
group.
\begin{proposition}
Let $G$ be an algebraic group. The matrix $B_G$ has the following properties:
\begin{enumerate}
\item The $(1,1)^{th}$ entry of $B_G$ is $\psi^{dim(\mathcal Z(G))}$. In fact, the diagonal entries 
of $B_G$ are $\psi^{dim(\mathcal Z(\mathcal Z_G(g_1, \ldots, g_r))}$ where $\mathcal Z(\mathcal 
Z_G(g_1,\ldots, g_r))$ is the center of $\mathcal Z_G(g_1, \ldots, g_r)$.
\item The entries in the first column are $\psi^{\dim(zcl(g))-\dim cl(g)}$ for various 
$z$-classes in $G$, and $0$ corresponding to $2$-tuples onwards. 
\item All entries in the first row, except first one, are $0$. 
\item Every row (second onwards) has a non-zero entry before the diagonal, i.e, for all $i>1$ 
there exists $i_0<i$ such that $(B_G)_{i,i_0}\neq 0$. 
\item If $\mathcal Z_G(g_1,\ldots, g_r)$ is Abelian, then the corresponding column has all 
entries $0$ except at the diagonal which is $\psi^{dim(\mathcal Z_G(g_1, \ldots, g_r))}$.
\item When $G$ is a reductive group, the branching matrix $B_G$ is a finite size matrix. 
\end{enumerate}
\end{proposition}
\begin{proof}
To prove (1) we note that a group $H$ is centralizer of its central elements. The central 
elements form a single $z$-class but distinct conjugacy classes. Thus, $(g_1, \ldots, g_r)$ 
appearing as a $z$-class in the group $\mathcal Z_G(g_1, \ldots, g_r)$ gives the following: 
$zcl((g_1, \ldots, 
g_r))= \mathcal Z(\mathcal Z_G(g_1, \ldots, g_r))$, and $cl((g_1, \ldots, 
g_r))= (g_1, \ldots, g_r)$. Hence the required result.

Proof of (2) is clear from the process to obtain $B_G$. Proof of (3) follows as the group $G$ 
itself can't appear as a subgroup of its proper centralizer. Proof of (4) follows from the 
process to obtain $B_G$, as a new row (and column) is added when a new centralizer type appears. 
Proof of (5) is clear.

The proof of (7) follows form  Proposition~\ref{finite-branches}.
\end{proof}
We require certain properties of the branching matrix $B_H$ when $H=\mathcal Z_G(g_1, \ldots, 
g_r)$ with respect to $B_G$. Recall the process of constructing $B_G$ as mentioned in the 
beginning of this subsection. All of the $z$-classes of $m$-tuples of $H$ are nothing but 
$z$-classes of $r+m$-tuples of $G$. Thus, to get $B_H$ we mark these rows and columns in $B_G$ and 
collect these entries in a new matrix. We warn here that the submatrix simply obtained from $B_G$ 
may not be the $B_H$, if we simply compute branching matrix for $H$ as per definition, since we 
have not fixed any strict order on the tuples. The matrix $B_H$ is a submatrix of $B_G$ consisting 
of those entries $(B_{G})_{a,b}$, where $a$ and $b$ occur in the list of branching of the class 
$\tau$ for $z$-classes $a, b$ of $\mathcal Z_G(\tau)$ of tuples. We have the following,
\begin{proposition}\label{BM-subgroup}
Let $(g_1,\ldots, g_r)$ be an $r$-tuple representing a $z$-class in a reductive group $G$. Then, 
\begin{enumerate}
\item the branching matrix $B_H$ of $H:=\mathcal Z_G(g_1, \ldots, g_r)$ is a submatrix of 
$B_G$. 
\item Let $\tau$ be a branch of $H$. Then, $(B_H^d)_{a\tau}= B_{\mathcal 
Z_H(\tau)}^d$.
\end{enumerate} 
\end{proposition}
\begin{proof}
The part (1) is clear from the explanation given above. 

Now to prove (2) we note that given $\tau$, a branch of $H$, the branching submatrix 
$B_{\mathcal{Z}_H(\tau)}$ of $B_H$ consists of $\tau$, branches of $\tau$, and the branches of 
those branches of $\tau$ and so on. When $a$ is a branch of $\tau$, or a branch of a branch of 
$\tau$, we 
see that,
$(B_H^2)_{a\tau} = \sum_{\eta}(B_H)_{a\eta}(B_H)_{\eta\tau}$. Now, $(B_H)_{a\eta}(B_H)_{\eta\tau} 
\neq 0$ if and only if $\eta$ is a branch of $\tau$, and  $a$ is a branch of $\eta$. Hence 
$(B_H)_{a\eta}(B_H)_{\eta\tau}$ is non-zero only if both $a$ and $\eta$ are in the branching 
submatrix, $B_{\mathcal{Z}_H(\tau)}$ of $\tau$. Hence, $(B_H^2)_{a\tau} = 
\sum_{\eta}(B_{\mathcal{Z}_H(\tau)})_{a\eta}(B_{\mathcal{Z}_H(\tau)})_{\eta\tau}=(B_{\mathcal 
Z_H(\tau)}^2)_{a\tau}$.

Now, we complete the proof by induction. Let us assume the equation is true up to $d$, and prove 
it for $d+1$. 
\begin{equation*}
\begin{aligned}
(B_H^{d+1})_{a\tau} &= \sum_{\eta}(B_H)_{a\eta}(B_H^d)_{\eta \tau} = 
\sum_{\eta}(B_H)_{a\eta}(B_{\mathcal{Z}_H(\tau)}^d)_{\eta\tau} \text{ by induction}\\
&= \sum_{\eta}(B_{\mathcal{Z}_H(\tau)})_{a\eta}(B_{\mathcal{Z}_H(\tau)}^d)_{\eta\tau} = 
(B_{\mathcal{Z}_H(\tau)}^{d+1})_{a\tau}.
\end{aligned}
\end{equation*}
This completes the proof.
\end{proof}
\noindent Note that, when $G$ is reductive, $H:=\mathcal Z_G(g_1, \ldots, g_d)$ has finite size 
branching matrix even though $H$ may not be reductive. 

\section{Dimension of commuting tuples}\label{section-dim-ct}

Whether the variety of commuting $d$-tuples, $\C_d(G)$, is an irreducible variety is an 
active topic of research. For a simply connected semisimple algebraic group $G$, 
Richardson~\cite[Theorem C]{Ri} proved that $\C_2(G)$ is an irreducible variety. However, for our 
work we need to only understand the dimension of this variety. Clearly, when $G$ is Abelian, 
$\dim\C_d(G)=d\dim(G)$. We relate the dimension of $\C_d(G)$ with computation of $d$-th power of 
the branching matrix $B_G$ here. Recall, that the entries of $B_G$ are monomials in $\psi$ and 
hence the entries of $B_G^d$ will be polynomials in $\psi$.
\begin{proposition}\label{bgd}
Let $G$ be a reductive algebraic group, and let $H=\mathcal Z_G(g_1, \ldots, g_r)$ be the 
centralizer of an $r$-tuple. Then for $d\geq 1$, 
$$\deg(\mathbf{1}.B_H^d.e_1)\leq \dim(\C_d(H)) \leq \deg(\mathbf{1}.B_H^d.e_1) + \dim(H)$$ 
where $\mathbf{1}$ is a row matrix with all $1$'s and $e_1$ is a column matrix with first entry 
$1$ and all others $0$.
\end{proposition}
\begin{proof}
First, we prove $\dim(\C_d(H)) \leq  \deg(\mathbf{1}.B_H^d.e_1) + \dim(H)$.
We will use double induction on $r$ and $d$. First we prove this for $d=1$. In this case, the 
required statement would be
$$\dim(H) \leq  \deg(\mathbf{1}.B_H.e_1) + \dim(H)$$ 
which is trivially true for any $H$. Let us assume induction up to $d$. Before going ahead, we 
recall the following from Proposition~\ref{BM-subgroup}. From the branching matrix $B_G$, we can 
obtain the branching matrix for any of the centralizer subgroup $\mathcal Z_G(\tau)$ where 
$\tau$ is a $z$-class of tuples. Further, the branching matrix $B_{\mathcal Z_G(\tau)}$ is a 
submatrix of $B_G$ consisting of those entries $(B_{G})_{ab}$, where $a$ and $b$ occur in the list 
of branching of the class $\tau$ for $z$-classes $a, b$ of $\mathcal Z_G(\tau)$ of tuples. Also, 
from Proposition~\ref{BM-subgroup} we note that when $B_G$ is multiplied with itself, this 
submatrix multiplies only with itself. Now, we write,
$$\C_{d+1}(H) = \bigcup_{\tau} zcl(\tau) \times \C_d(\mathcal Z_H(\tau)),$$
where the union runs over $z$-classes in $H$. Let us denote the dimension of $H$ by $n$, and that of $\mathcal Z_H(\tau)$ by $n_{\tau}$. 
So, we have
\begin{eqnarray*}
&& \dim\C_{d+1}(H) = \max_{\tau} \left\{\dim zcl(\tau) + \dim \C_d(\mathcal 
Z_H(\tau))\right\} \\ 
 &\leq& \max_{\tau} \left\{\left(\deg (B_{H})_{\tau 1} +\dim cl(\tau) \right) + \deg \left(\mathbf{1}.B_{\mathcal 
Z_H(\tau)}^d.e_{\tau} \right) + n_{\tau} \right\} ~\text{by induction}\\
 &= & \max_\tau \left\{\deg (B_{H})_{\tau 1} + n - n_{\tau} + \deg \left(\mathbf{1}.\left 
((B_{H})_{uv}\right)^d.e_\tau \right) + n_{\tau}  \right\}~\text{ types $u$, $v$ are } 
\\ && \text{branches of type $\tau$}\\
 &= & n + \max_\tau \left\{\deg(B_{H})_{\tau 1} + \deg \left(\sum_a (B_H^d)_{a\tau}\right) \right\} 
~\text{from Proposition~\ref{BM-subgroup}}\\
 &= & n+ \max_{\tau}\left\{ \deg\left(\sum_a (B_H^d)_{a\tau}.(B_{H})_{\tau 1}\right) \right\} = n+ 
\deg\sum_a (B_H^{d+1})_{a1}\\
&=& \deg (\mathbf{1}.B_H^{d+1}.e_1) + n.
\end{eqnarray*}
Here $e_t$ is the column matrix with $1$ at the $t^{th}$ place and $0$ elsewhere. 
This completes the proof of the right side inequality. 

Now, we need to prove $\deg(\mathbf{1}.B_H^d.e_1)\leq \dim(\C_d(H))$. We follow the notation set 
above and prove it by induction. To begin with, for $d=1$, we need to show 
$\deg(\mathbf{1}.B_H.e_1)\leq \dim(H)$. This follows as the left hand side is maximal possible $\dim 
zcl(g) - \dim cl(g)$, where $g\in H$, and $cl(g)\subset zcl(g) \subset H$. Now let us assume this 
for $d$ and prove for $d+1$. 
\begin{eqnarray*}
&& \dim\C_{d+1}(H) = \max_{\tau} \left\{\dim zcl(\tau) + \dim \C_d(\mathcal 
Z_H(\tau))\right\} \\ 
 &\geq& \max_{\tau} \left\{\left(\deg (B_{H})_{\tau 1} +\dim cl(\tau) \right) + \deg 
\left(\mathbf{1}.B_{\mathcal Z_H(\tau)}^d.e_{\tau} \right) \right\} ~\text{by induction}\\
 &\geq & \max_\tau \left\{\deg (B_{H})_{\tau 1} + \deg \left(\mathbf{1}.\left 
((B_{H})_{uv}\right)^d.e_\tau \right)  \right\} \\
&=&  \max_\tau \left\{\deg(B_{H})_{\tau 1} + \deg \left(\sum_a (B_H^d)_{a\tau}\right) \right\} \\
 &= &  \max_{\tau}\left\{ \deg\left(\sum_a (B_H^d)_{a\tau}.(B_{H})_{\tau 1}\right) \right\} 
= deg\sum_a (B_H^{d+1})_{a1} = \deg (\mathbf{1}.B_H^{d+1}.e_1).
\end{eqnarray*}
This completes the proof.
\end{proof}

\noindent Thus, we can rewrite
\begin{equation}\label{equation1}
\frac{\deg(\mathbf{1}.B_G^{d}.e_1)}{d\dim G} \leq cp_d(G) \leq \frac{\deg(\mathbf{1}.B_G^{d}.e_1) + 
\dim(G)}{d\dim G} = \frac{\deg(\mathbf{1}.B_G^{d}.e_1)}{d\dim G} + \frac{1}{d}.
\end{equation}
In the next section, we compute this for reductive algebraic groups.

\section{Commuting probability for reductive algebraic groups}\label{section-alg-gps}

Let $K$ be an algebraically closed field and $G$ be a reductive algebraic group over $K$ of 
dimension $n$. In this section, we discuss the asymptotic value of the commuting 
probabilities for $G$. In~\cite{Ga}, it is proved that $cp_2(G)=\frac{n + \rho}{2n}$ where $\rho$ 
is the rank of $G$. Using the argument there, one can show that $cp_d(G)\geq \frac{n + 
(d-1)\rho}{dn}$. As noticed in~\cite{KPP} for finite groups while studying asymptotic behavior 
of $cp_d(G)$ as $d$ gets large, we see that the maximal dimension of an Abelian subgroup 
plays a role here. Henceforth, whenever we talk about maximal Abelian subgroup, we mean a subgroup 
of  maximal size/dimension among Abelian subgroups. Our main theorem is as follows:
\begin{theorem}\label{theorem-cpd-alg-gps}
Let $G$ be a reductive algebraic group over an algebraically closed field $K$. Let $\dim(G)=n$, 
maximal dimension of an Abelian subgroup be $\alpha$ (in general, $\alpha\geq rank(G)$), and 
the size of the branching matrix be $\beta$. Then, for large enough $d$,
$$\left(1-\frac{\beta}{d}\right)\frac{\alpha}{n}=\frac{(d-\beta)\alpha}{dn} \leq cp_d(G) \leq  
\frac{\alpha}{n} +\frac{1}{d}.$$ 
Thus, as $d$ gets large, the commuting probabilities $cp_d(G)\sim \frac{\alpha}{n}$.
\end{theorem}
\noindent We need a couple of Lemmas before we prove this result.
\begin{lemma}\label{LemMax1-alg}
Let $G$ be a reductive algebraic group, and $B_G$ be its branching matrix. Then, the maximal entry 
of the branching matrix $B_G$ is $\psi^{\alpha}$.
\end{lemma}
\begin{proof}
Let $(g_1, \ldots, g_k)$ be a commuting $k$-tuple of $G$, and suppose the common centralizer 
$\mathcal Z_G(g_1,\ldots, g_k)$ is Abelian of order $b$. For each $1 \leq i \leq k$, let 
$Z_i = \mathcal Z_G(g_1, \ldots, g_i)$. Thus, we have $Z_{i+1}=\mathcal Z_{Z_i}(g_{i+1})$, for 
$1\leq i \leq k-1$. We get a non-increasing sequence of centralizer subgroups 
$Z_1\supset  Z_2 \supset \cdots \supset Z_{k-1}\supset Z_k$. Now, we claim that the 
centres of these centralizers $\mathcal Z(Z_i)$ form a non-decreasing sequence, i.e., $\mathcal 
Z(Z_{i})\subset \mathcal Z(Z_{i+1})$. For this, let $z \in \mathcal Z(Z_i)$. Now, as $z$ commutes 
with $g_{i+1}$, we have $z \in Z_{i+1}$. But, $Z_{i+1}\subset Z_i$ hence $z \in \mathcal 
Z(Z_{i+1})$. This proves $\mathcal Z(Z_{i+1})\supset \mathcal Z(Z_i)$. 

Let $g\in H$, which is a common centralizer of a commuting tuple. Let $k$ be the smallest integer 
such that $(g=g_1, \ldots, g_k)$ be a commuting $k$-tuple of $H$, with the common centralizer 
$\mathcal Z_H(g_1,\ldots, g_k)$, Abelian of dimension $b$. We claim that
$$\dim zcl(g)-\dim cl(g) \leq b.$$
Since $zcl(g)=\bigcup_{t} cl(t)$ where union is over $t\in H$ of which centralizer is conjugate 
to $\mathcal Z_H(g)$ (see Remark~\ref{computingBG}). Thus,
\begin{eqnarray*}
&& \dim zcl(g)-\dim cl(g) \\
&=& -\dim N_H(\mathcal Z_H(g)) + \dim \{x\in cl(H) \mid \mathcal Z_H(x)= \mathcal Z_H(g)\} + \dim 
\mathcal Z_H(g)\\
&\leq & \dim \{x\in cl(H) \mid \mathcal Z_H(x)= \mathcal Z_H(g)\} \leq \dim \mathcal 
Z(\mathcal Z_H(g))\\
&\leq& \dim \mathcal Z(\mathcal Z_H(g_1,\ldots, g_k))~(\text{from the first para of this proof})\\
&=&b \leq \alpha.
\end{eqnarray*}
For the last but one line, we use the following: $ \{x\in H \mid \mathcal Z_H(x)= 
\mathcal Z_H(g)\} \subset \mathcal Z(\mathcal Z_H(g))$. 

\end{proof}
Now we prove a result regarding entries of power of a matrix which we will apply to the branching 
matrix.
\begin{lemma}\label{Lempoly-alg}
Let $B=(b_{i,j})$ be a non-negative $m\times m$ matrix with all diagonal entries non-zero. Suppose, $B$ has the property that every row has a non-zero entry before the diagonal, i.e., $\forall i >1$ there exists $i_0 <  i$ such that $b_{i,i_0}\neq 0$. Then, for $r\geq  2$, the $(l,1)^{th}$ entry of $B^r$ is a polynomial in $b_{l,l}$ of degree at most $r$, and at least $r-m$. 
\end{lemma}
\begin{proof}
Let $B=(b_{i,j})$. Then,
\begin{eqnarray*}
(B^r)_{s,1} &=& \sum_{l_1}\sum_{l_2} \cdots \sum_{l_{r-1}} b_{s,l_1} b_{l_1,l_2}\cdots 
b_{l_{r-1},1}\\
&=& b_{s,1}b_{1,1}^{r-1} + b_{s,s}^{r-1}b_{s,1}+\cdots
\end{eqnarray*}

If $l = 1$ and $b_{i,1}=0$ for all $i>1$, then $B = \begin{pmatrix} b_{1,1} & * \\ 0 & C \end{pmatrix}$, where $C$ is an $(m-1)\times (m-1)$ matrix. For any $r$, we can see that $(B^r)_{1,1}= b_{1,1}^r$. 

Let $b_{l,1} \neq 0$. Then, $(B^2)_{l,1} = \sum_{i=1}^m b_{l,i}b_{i,1} = b_{l,1}b_{1,1} + \cdots + b_{l,l}b_{l,1} + \cdots$, which we rewrite as $b_{l,l}b_{l,1} + b_{l,1}b_{1,1}  + d_2$, where $d_2$ denotes the rest of the terms (which are constant in $b_{l,l}$ and $b_{1,1}$). Now, $(B^3)_{l,1} = (B.B^2)_{l,1} = \sum b_{l,i}(B^2)_{i,1}=b_{l,1}b_{11}^2 + b_{l,l}.(b_{l,1}b_{1,1} + b_{l,l}b_{l,1} + d_2) + d_3 = b_{l,l}^2b_{l,1} + b_{l,l}.b_{l,1}b_{1,1} + b_{l,1}b_{11}^2   + d_4$. The result follows by induction, as we can  prove that $(B^r)_{l,1} = b_{l,1}b_{l,l}^{r-1} + c_1 b_{1,1}b_{l,l}^{r-2} + \cdots + c_{r-1}b_{1,1}^{r-1} + d_r$, where $d_r$ denotes the rest of the 
terms in the sum. 

Now, suppose $b_{l,1} = 0$ (obliviously $l >1$). Let $u \leq m$ be smallest such that we have a 
sequence of numbers $l=k_1, k_2, \ldots, k_{u}$ such that the entries $b_{l, k_{2}}, \ldots, 
b_{k_{u-1},k_u}, b_{k_u,1}$ are all non-zero. Then, $(B^u)_{l,1} = b_{l,k_{2}}b_{k_{2}k_{3}}\cdots 
b_{k_u,1} + \cdots$ is non-zero. Now, following the argument similar to the last para, we see that 
for $r > u$, we have $(B^r)_{l,1}$ is a polynomial in $b_{l,l}$ of degree  $r-u$. Finally, a 
sequence of numbers $l=k_1, k_2, \ldots, k_{u}$ with the required property is guaranteed because of 
 the given condition as follows. Begin with the $l^{th}$ row, and find smallest $k_2<l$ such that 
$b_{l,k_2}\neq 0$. Next, look at the row $k_2$ and find smallest $k_3< k_2$ such that $b_{k_2, 
k_3}\neq 0$. We will be done when we get $k_u$ with $b_{k_u,1}\neq 0$ (and noting that $b_{2,1}\neq 
0$).  
\end{proof}

\noindent Now, we prove the theorem. 
\begin{proof}[\bf Proof of the Theorem~\ref{theorem-cpd-alg-gps}]
In view of Proposition~\ref{bgd} and Equation~\ref{equation1}, we require to prove the following 
for large enough $d$, 
$$(d-\beta)\alpha \leq \deg(\mathbf{1}.B_G^d.e_1)\leq  d\alpha.$$
Note that entries of $B_G$ are either $0$ or powers of $\psi$, which follow the condition required in the Lemma~\ref{Lempoly-alg}. Write $B_G=(\psi^{x_{i,j}})$. Then for $d\geq 2$,
$$(B^d)_{s,1} = \sum_{l_1}\sum_{l_2} \cdots \sum_{l_{d-1}} b_{s,l_1} b_{l_1,l_2}\cdots 
b_{l_{d-1},1}=\sum_{l_1}\sum_{l_2} \cdots \sum_{l_{d-1}} \psi^{x_{s,l_1} +x_{l_1,l_2}\cdots +
x_{l_{d-1},1}}$$
gives us that $\deg (B^d)_{s,1} \leq r\alpha$. Thus, $\deg(\mathbf{1}.B_G^d.e_1)$ which is the degree of sum of the first column $\leq  d\alpha$.

From Lemma~\ref{Lempoly-alg}, we note that $(B^d)_{s,1}$ is a polynomial in $\psi^{x_{s,s}}$ of degree at least $d-\beta$, i.e., $\deg(B^d)_{s,1} \geq (d-\beta)x_{s,s}$. The largest degree on diagonal (in fact whole of $B_G$) is $\alpha$. Thus, $\deg(\mathbf{1}.B_G^d.e_1) \geq (d-\beta)\alpha$.

Hence, $$cp_d(G)\leq \frac{\deg(\mathbf{1}.B_G^d.e_1)}{d\dim G} +\frac{1}{d} \leq \frac{\alpha}{n} +\frac{1}{d}.$$
Also, 
$$cp_d(G)\geq \frac{\deg(\mathbf{1}.B_G^d.e_1)}{d\dim G} \geq  \frac{(d-\beta)\alpha}{dn} = \left(1- 
\frac{\beta}{d}\right)\frac{\alpha}{n}.$$
This completes the proof.
\end{proof}

\section{Asymptotic value of commuting probabilities in finite groups}\label{section-fg}

Let $G$ be a finite group, and $d\geq 2$ be a positive integer. We begin with recalling the 
relation between commuting probabilities and simultaneous conjugacy classes of commuting tuples 
via branching matrix $B_G$. The entries of this matrix represent the number of conjugacy classes 
of tuples which are in the same $z$-class, i.e, the size of $z$-class divided by the size 
of conjugacy class. This is done in~\cite{SS}, and we refer to the same for various definitions 
and terminologies used in this section. Some of these ideas have been generalized in the earlier 
sections for algebraic groups. Since $G$ acts on the set $\C_d(G)$, by component-wise conjugation, 
we have simultaneous conjugacy classes (of commuting $d$-tuples). Let $c_G(d)$ denote the number 
of orbits in the above action, also called simultaneous conjugacy classes of 
$d$ tuples in $G$. It has been proved in~\cite[Theorem 1.1]{SS} that,
\begin{theorem}\label{branchingmatrix}
Let $G$ be a finite group and $d\geq 2$, an integer. Let $B_G$ be the branching matrix of $G$. 
Then,  
$$cp_d(G) = \frac{c_G(d-1)}{|G|^{d-1}}=  \frac{\mathbf{1}.B_G^{d-1}.e_1}{|G|^{d-1}}$$ 
where $\mathbf{1}$ is a row matrix, with all $1$'s, and $e_1$ is a column matrix with first entry 
$1$, and $0$ elsewhere.
\end{theorem}
\noindent The commuting probabilities $cp_d(G)$, for $d=2,3,4, 5$ have been explicitly calculated 
in~\cite{SS2} 
using the corresponding branching matrices with the help of SageMath~\cite{Sagemath} for the 
following classical groups over a finite field $\Fq$ (where $q$ is odd): $G= GL_2(\Fq)$, 
$U_2(\Fq)$, $GL_3(\Fq)$, $U_3(\Fq)$, and $Sp_2(\Fq)$. The data obtained from these groups led us to explore the asymptotic behavior of $cp_d(G)$ as $d$ gets large for a fixed $G$. In other words, we would like to understand what is $cp_d(G)$ asymptotic to, as a function of $d$? Interestingly, Kaur, Prajapati and Prasad~\cite[Theorem 3.1]{KPP} have shown that for a finite group $G$ and positive integer $d$, 
the number $c_G(d)$ is asymptotic to $a^d$, up to multiplication by a positive constant, where $a$ 
denotes the maximal size of an Abelian subgroup of $G$. That is, there exist a positive 
integer $m$ 
so that $c_G(d) \sim m a^d$. Thus, from Theorem~\ref{branchingmatrix}, it follows that,
\begin{equation}\label{eq-kpp}
cp_{d}(G) \sim m \left(\frac{a}{|G|}\right)^{d-1}.
\end{equation} 
Using the ideas in Section~\ref{section-alg-gps}, we give an alternate proof of this.

We are going to make use of the branching matrix $B_G$, for the finite group $G$, as described 
in~\cite{SS, SS2} and the Theorem~\ref{branchingmatrix}. For $d \geq 2$, the size of the set 
$\C_d(G)$ of commuting $d$-tuples of elements of $G$ is 
 $$|\C_d(G)| = cp_d(G).|G|^d = |G| \left(\mathbf{1}.B_G^{d-1}.e_1 \right).$$
Thus, to understand $\C_d(G)$ we need to understand the entries of the first column in the matrix 
$B_G^{d-1}$. We begin with a result for finite groups similar to Lemma~\ref{LemMax1-alg} proved 
earlier for algebraic groups.
\begin{lemma}\label{LemMax1}
With the notation as above, the maximal entry in the branching matrix $B_G$ of $G$ is the 
maximal size of an Abelian subgroup $a$.
\end{lemma}
\begin{proof}
Proceeding along the lines of the proof of Lemma~\ref{LemMax1-alg}, we note that the entries of the matrix $B_G$ will satisfy the following:
$$
\frac{|zcl(g)|}{|cl(g)|} = \frac{|\mathcal Z_H(g)|}{|N_H(\mathcal Z_H(g))|}.|\{x\in cl(H)\mid \mathcal Z_H(x)=\mathcal Z_H(g)\}| 
\leq  |\mathcal Z(\mathcal Z_H(g))| \leq a$$
where $H$ is a centralizer of some tuples. The result follows.
\end{proof}
Now, we give an alternate proof of the Equation~\ref{eq-kpp}. The proof is along the same lines as that of Theorem~\ref{theorem-cpd-alg-gps} and hence we keep it brief. 
\begin{proposition}\label{prop-finite-group}
Let $G$ be a finite group and $a$ be the size of maximal Abelian subgroup. Then, for large enough 
$d$, the size of commuting $d$-tuples, $|\C_d(G)| \sim m|G|a^{d-1}$ where $m$ is a 
constant.
\end{proposition} 
\begin{proof}
We begin with proving that there exists a constant $m$ such that $ \mathbf{1}.B_G^{d}.e_1= m a^{d} 
+ \mathcal O(a^{d-1})$ when $d$ is large. Note that, the branching matrix  $B_G=(b_{i,j})$ 
satisfies the properties required in the Lemma~\ref{Lempoly-alg}. Thus, $(B_G^d)_{s,1}$ is a 
polynomial in $b_{s,s}$ of degree at most $d$ and at least $d-\beta$ where $\beta$ is the size of 
$B_G$. From Lemma~\ref{LemMax1}, the largest diagonal entry (in fact, the largest entry) of $B_G$ 
is $a$. Thus, for large $d$, we get $\mathbf{1}.B_G^{d}.e_1= ma^d + \mathcal O(a^{d-1})$ where $m$ 
is a constant depending on $G$ only.

Now we have, $|\C_d(G)| = |G| \left( \mathbf{1}.B_G^{d-1}.e_1 \right)$, thus,  
$$|\C_d(G)|= m |G| a^{d-1} + \mathcal O(a^{d-2}).$$
This proves the required result.
\end{proof}
\noindent Next we look at some examples.

\subsection{Application to finite reductive groups}
Let $\mathbb F_q$ be a finite field and $K$ its algebraic closure. 
Let $G$ be a connected reductive group over $K$, with Frobenius map $F$ so that $G(\mathbb F_q) = G^F$ is a finite group of Lie type. Then, we have the following,
\begin{theorem}\label{CP-fgl}
Let $G$ be a connected reductive group defined over a finite field $\mathbb F_q$. Let us denote the $\mathbb F_q$ points of $G$ by $G(\mathbb F_q)=G^F$. Then, for large enough $q$,
$$|\C_d(G(\mathbb F_q))| \sim q^{n+ (d-1)\alpha}$$
up to a constant where $n$ is the dimension of $G$, and $\alpha$ is the maximal dimension of 
an Abelian subgroup. Hence up to a constant, $cp_d(G(\mathbb F_q))\sim 
\left(\frac{1}{q^{n-\alpha}}\right)^{d-1}$. 
\end{theorem}
\begin{proof}
From~\cite[Theorem 11.16]{St2}, it follows that $|G(\mathbb F_q)|=q^n+\mathcal O(q^{n-1})$. Thus, when $q$ is large enough (to ensure that only its power is dominating), we get the result from Proposition~\ref{prop-finite-group}.
\end{proof}

Maximal size/dimension of Abelian subgroups are well studied for finite classical groups and 
more generally for finite simple groups (see~\cite{Vd1, Vd2, Wo1, Wo2, Ba}). It turns out that for 
$G$, a finite simple group of Lie Type of large enough rank, an 
Abelian subgroup of maximal order is unipotent. If $G$ is not simple, then an Abelian subgroup of 
maximal order is the product of the centre of the group and an Abelian unipotent group in $G$ of 
maximal order. Now we look at some examples, mainly of finite groups of Lie type where maximal sized
Abelian subgroups are known, and give asymptotic value of commuting probabilities as $d$ gets 
large. In what follows, we take $q$ large enough and use the formula $\left(\frac{a}{|G|}\right)^{d-1}$ to compute the asymptotic value of $cp_d(G)$. 

\begin{example}\label{ExGL2}
For the group $GL_2(q)$, we have $|GL_2(q)| = (q^2-1)(q^2-q)$ and the maximal order of an Abelian 
subgroup is $q^2-1$ (given by an anisotropic torus). Then, 
$cp_d(GL_2(q))\sim \left(\displaystyle\frac{1}{q(q-1)}\right)^{d-1}$ up to a constant.

For $G = GL_3(q)$, we have the maximal size of an Abelian subgroup $a = q^3-1$ (again given by 
an  
anisotropic torus), and $|G| = (q^3-1)(q^3-q)(q^3-q^2)$. Then, 
$$cp_d(GL_3(q))\sim \left(\displaystyle\frac{1}{q^3(q^2-1)(q-1)}\right)^{d-1}.$$
In both of these cases, the maximal Abelian is obtained by centralizer of a regular semisimple 
element, that is, by a $1$-tuple.
\end{example}
\begin{example}
Consider the group $GL_{2l}(q)$ and $l\geq 2$. The following block diagonal matrices 
$$A=\left\{ \begin{pmatrix} \lambda I_l & X \\ &\lambda I_l \end{pmatrix}\mid X\in M_l(q), 
\lambda\in \mathbb F_q^* \right\}$$ 
give a maximal sized Abelian subgroup with order $(q-1)q^{l^2}= q^{l^2+1} +\mathcal O(q^{l^2})$. 
Notice 
that a maximal torus is of size $q^{2l}+\mathcal O(q^{2l-1})$ and  $A$ is bigger than this. 
Further, we note that $A$ can be obtained as a centralizer of commuting $3$-tuple as follows:
$$A= \mathcal Z_{GL_{2l}(q)}\left( \begin{pmatrix} I & I \\ & I \end{pmatrix}, \begin{pmatrix} I 
& \Lambda \\ & I \end{pmatrix}, \begin{pmatrix} I & N \\ & I \end{pmatrix}\right)$$
where $\Lambda = \begin{pmatrix} \lambda_1 & & \\ & \ddots 
&\\ & & \lambda_l \end{pmatrix}$ with all distinct entries, and $N = \begin{pmatrix} 0 & &\cdots 
&1\\ 1& 0 &\cdots& \\\vdots &\ddots&\ddots&\\ 0& & 1&0 \end{pmatrix}$.

Now, consider the group $GL_{2l+1}(q)$ for $l\geq 2$ and let $A$ be the following block diagonal 
matrices: $$A=\left\{ \begin{pmatrix} \lambda I_l & X \\ &\lambda I_{l+1} \end{pmatrix}\mid X\in 
M_{l\times (l+1)}(q), \lambda\in \mathbb F_q^* \right\}.$$ 
Then, $A$ is a maximal size Abelian subgroup with order $(q-1)q^{l(l+1)}= q^{l(l+1)+1} +\mathcal 
O(q^{l^2+l})$. Once again this can be obtained as a centralizer of commuting tuple.

Thus, for all $n \geq 4$, the maximal cardinality of any Abelian subgroup of $GL_n(q)$ is 
$q^{[n^2/4]}(q-1)$. Thus, up to a constant, the commuting probabilities 
$$cp_d(GL_{2l}(q))\sim \left(\frac{1}{q^{l(l-1)}\prod_{i=2}^{2l}(q^i-1)}\right)^{d-1}$$
and $$cp_d(GL_{2l+1}(\Fq))\sim \left(\frac{1}{q^{l^2}\prod_{i=2}^{2l+1}(q^i-1)}\right)^{d-1}.$$
\end{example}
  
\begin{example}\label{ExU2}
For $U_2(q)$, from~\cite[Proposition 3.3]{SS2} we have $cp_d(U_2(q)) = cp_d(GL_2(q))$ for all 
$d\geq 2$, so the asymptoticity is the same as in Example~\ref{ExGL2}.

For $U_3(q)$, the maximal size for an abelian subgroup is $(q+1)^3$. Thus, $cp_d(U_3(q))$ is 
asymptotic to $$\left(\frac{1}{q^3(q^2-q+1)(q-1)}\right)^{d-1}.$$

Now we take $U_n(q)$ for $n \geq 4$, its centre is of size $q+1$, and the maximal cardinality of 
its unipotent Abelian subgroup is $q^{[n^2/4]}$, like it is with $GL_n(q)$. Hence, the maximal 
Abelian cardinality is $q^{[n^2/4]}(q+1)$. Thus up to a constant, 
$$cp_d(U_{2l}(q))\sim \left(\frac{1}{q^{l(l-1)}\prod_{i=2}^{2l}(q^i-(-1)^i)}\right)^{d-1}$$
and
$$ cp_d(U_{2l+1}(q))\sim \left(\frac{1}{q^{l^2}\prod_{i=2}^{2l+1}(q^i-(-1)^i)}\right)^{d-1}.$$
We notice $q\leftrightarrow -q$, Ennola like duality, between the formula of $GL$ and $U$ for the 
asymptotic value.
\end{example}

\begin{example}\label{Symp2}
For $l\geq 1$ and $q$ odd, let us consider the symplectic group $Sp_{2l}(q)=\{g\in 
GL_{2l}(q) \mid \tr g \beta g =\beta\}$ where $\beta = \begin{pmatrix} & I_l \\-I_l & 
\end{pmatrix}$. 
For $l=1$ the group $Sp_2(q)\cong SL_2(q)$, and the maximal abelian subgroup is of size $2q$. So, 
for $d \geq 2$, the commuting probabilities $cp_d(Sp_2(q))$ is asymptotic (upto multiplication by 
some positive constant) to 
$$\left(\frac{2}{q^2-1}\right)^{d-1}.$$

Now, for large enough $q$,  
$$A=\left\{ \pm \begin{pmatrix} I_l & X \\ &  I_l \end{pmatrix}\mid \tr X = X \right\}$$ 
is a maximal size Abelian subgroup of $Sp_{2l}(q)$ of order $2 q^{\frac{l(l+1)}{2}}$. Hence, up to 
a 
constant,  
$$cp_d(Sp_{2l}(q))\sim  
\left(\frac{2q^{\frac{l(l+1)}{2}}}{q^{l^2}\prod_{i=1}^l(q^{2i}-1)}\right)^{d-1} = 
\left(\frac{2}{q^{\frac{l(l-1)}{2}}\prod_{i=1}^l(q^{2i}-1)}\right)^{d-1}.$$
\end{example}

\begin{example}
Let us consider the orthogonal group $O_{2l}(q)=\{g\in GL_{2l}(q) \mid \tr g \beta g =\beta\}$ 
where $\beta = \begin{pmatrix} & I_l \\ I_l & \end{pmatrix}$ and $q$ odd. Then,  
$$A=\left\{ \pm \begin{pmatrix} I_l & X \\ &  I_l \end{pmatrix}\mid \tr X = - X \right\}$$ 
is a maximal size Abelian subgroup with order $2 q^{\frac{l(l-1)}{2}}$. Hence up to a constant,  
$$cp_d(O_{2l}(q))\sim 
\left(\frac{2q^{\frac{l(l-1)}{2}}}{2q^{l(l-1)}\prod_{i=1}^l(q^{2i}-1)}\right)^{d-1 } = 
\left(\frac{1}{q^{\frac{l(l-1)}{2}}\prod_{i=1}^l(q^{2i}-1)}\right)^{d-1}.$$
\end{example}



\end{document}